\documentclass[12pt]{amsart}
\usepackage[margin=1in]{geometry}                
\usepackage{graphicx}
\usepackage{amssymb}
\newtheorem{theorem}{Theorem}
\newtheorem{definition}{Definition}

\newtheorem{lemma}{Lemma}

\newtheorem{proposition}{Proposition}
\newtheorem{corollary}{Corollary}

\newcommand{\CC}{\mathbb{C}}

\newcommand{\ZZ}{\mathbb{Z}}
\newcommand{\RR}{\mathbb{R}}
\newcommand{\NN}{\mathbb{N}}

\newcommand{\cals}{\mathcal{S}}

\newcommand{\cala}{\mathcal{A}}
\newcommand{\cale}{\mathcal{E}}

\newcommand{\calm}{\mathcal{M}}
\newcommand{\one}{\mathbf{1}}
\newcommand{\sgamma}{\langle s \rangle^\gamma}
\newcommand{\supp}{{\hbox{supp}\,}}
\newcommand{\tsigma}{\sigma^0_{j,a}}
\newcommand{\half}{\frac{1}{2}}

\title{Maximal operators and decoupling for $\Lambda(p)$ Cantor measures} 
\author{Izabella {\L}aba}				
\date{\today}

\begin{document}

\begin{abstract}
For $2\leq p<\infty$, $\alpha'>2/p$, and $\delta>0$,
we construct Cantor-type measures on $\RR$ supported on sets of Hausdorff dimension $\alpha<\alpha'$
for which the
associated maximal operator is bounded from $L^p_\delta (\RR)$ to $L^p(\RR)$.
Maximal theorems for fractal measures on the line
were previously obtained by {\L}aba and Pramanik \cite{LP-diff}. The result here is weaker in that we are not
able to obtain $L^p$ estimates; on the other hand, our approach allows Cantor measures that are 
self-similar, have arbitrarily low dimension $\alpha>0$, and have no Fourier decay.
The proof is based on a decoupling inequality similar to that of {\L}aba and Wang \cite{Laba-Wang}.


\end{abstract}

\maketitle


\section{introduction}


\subsection{The Cantor set constuction}\label{cantor-construct}

We define a Cantor set $E\subset\RR$, with the associated measure $\mu$ supported on it, as follows. 
Let $N$ be a large positive integer, and let $N_0\in\mathbb{N}$ with $0<N_0<N$.
Let $\Sigma$ be a non-empty collection of subsets $S$ of $[N]:=\{0, 1,\dots, N-1\}$ such that
$|S|=N_0$ for all $S\in\Sigma$. Choose $S_1\in\Sigma$, and let
$$A_1={1}+N^{-1}S_1, \ \ E_1= A_1+[0,N^{-1}].$$
For each $a\in A_1$, choose $S_{2,a}\in\Sigma$ and let 
$$
A_{2,a}=a+N^{-2}S_{2,a},\ \ 
A_2=\bigcup_{a\in A_1} A_{2,a},\ \ E_2=A_2+[0,N^{-2}].
$$
Continuing by induction, let $k\geq 2$, and suppose that $A_j$ and $E_j$, $j=1,2,\dots,k$, have been constructed.
For every $a\in A_k$, choose $S_{k+1,a}\in\Sigma$, and let
$$
A_{k+1,a}=a+N^{-k-1}S_{k+1,a},\ \ 
A_{k+1}=\bigcup_{a\in A_k} A_{k+1,a},\ \ E_{k+1}=A_{k+1}+[0,N^{-k-1}].
$$
This yields a sequence of sets $[1,2] \supset E_1\supset E_2\supset E_3\supset\dots$, where each $E_j$ consists of 
$N_0^j$ intervals of length $N^{-j}$. For each $j$, let
$$
\mu_j=\frac{1}{|E_j|} \one_{E_j}.$$
We will identify the functions $\mu_j$ with the absolutely continuous measures $\mu_j\,dx$. It is easy to see that $\mu_j$ converge weakly as $j\to\infty$ to a probability measure $\mu$ supported on the Cantor set $E_\infty=\bigcap_{j=1}^\infty E_j$, and that $E_\infty$ has Hausdorff and Minkowski dimensions both equal to
$\alpha:=\frac{\log N_0}{\log N}$ (so that $N_0=N^\alpha$).
Furthermore, there is a constant $C_\mu>0$ such that for all $x\in \supp\mu$ we have
\begin{equation}\label{main-e1b}
C_\mu^{-1} r^\alpha \leq \mu((x-r,x+r))\leq C_\mu r^\alpha \ \ \ \forall \ r>0.
\end{equation}

We are particularly interested in the self-similar case, with $S_1=S_{j,a}=S$ for a fixed $S\in\Sigma$
and all $j,a$ in the construction. Then $\mu$ is a self-similar measure supported on the set
$$
E_\infty =\left\{x\in[1,2]:\ x={1}+\sum_{j=1}^\infty x_jN^{-j}, \ x_j\in S\hbox{ for all } j\in\NN \right\},
$$
and has similarity dimension $\alpha$. However, self-similarity is not required for our proof.
Our assumptions could be weakened further: for example, the same argument works (with appropriately modified constants) if the assumption that
$|S_{j,a}|=N_0$ for all $j,a$ is replaced by the weaker condition $c^{-1}N_0 \leq |S_{j,a}|\leq cN_0$ for some $c>0$, as long as
$S_{j,a}$ continue to be $\Lambda(p)$-sets. It should also be possible to allow constructions with slowly
varying parameters as in \cite{Laba-Wang}.


\subsection{$\Lambda(p)$ sets}

Our Cantor digit set $A$ will be provided by a theorem of Bourgain on $\Lambda(p)$ sets (\cite{bourg-lambdaP}; see also Talagrand \cite{talagrand}).

\begin{theorem}\label{bourgain-thm} (Bourgain \cite{bourg-lambdaP})
Let $p>2$.
For every $N\in\NN$ sufficiently large, there is a set $S\subset [N]$ of size $|S|\geq c_0 N^{2/p}$ such that
for any set of coefficients $\{c_a\}_{a\in S}$ we have
\begin{equation} \label{lambda-p-0}
\Big\|\sum_{a\in S} c_a e^{2\pi ia x}\Big\|_{L^{p}([0,1])}\leq C(p) \Big(\sum_{a\in S}|c_a|^2\Big)^{1/2},
\end{equation}
with $c_0$ and $C(p)$ independent of $N$.
\end{theorem}

The sets $S$ in Theorem \ref{bourgain-thm} are called  
$\Lambda(p)$ sets. It is well known (see \cite{bourg-lambdaP}) that Bourgain's lower bound on $|S|$ is optimal,
so that we must in fact have
\begin{equation}\label{largeA}
c_0 N^{2/p}\leq |S| \leq c_1 N^{2/p}
\end{equation}
with the constant $c_1$ independent of $N$. 
%
For convenience, we will always assume that $S\subset [N-1]$, i.e.\ $N-1\notin S$. This can always be arranged by removing $N-1$ from $S$ and adjusting the constants if necessary.


\subsection{Main result}
We define the maximal operator with respect to a probability measure
$\mu$:
\begin{equation} \label{max-e1}
{\mathcal M} f(x) :=  \sup_{t > 0} \int \left| f(x - ty) \right| d\mu(y)
=\sup_{t>0} \big[ \cala_t |f| \big] (x), \ \ f\in\cals,
\end{equation}
where 
$$\cala_tf := \int f(x-ty)\, d\mu(y) = \int \widehat{f}(\xi) \widehat{\mu}(t\xi) e^{2\pi i x\xi} d\xi .$$

Our main result is a bound on $\mathcal{M}$ when $\mu$ is a Cantor measure with
$\Lambda(p)$ digit sets.  We first specify rigorously the class of measures under consideration.

\begin{definition}\label{def-cantorsets}
We say that $E_\infty=\bigcap_{j\in\NN} E_j \subset\RR$ is a \emph{ $\Lambda(p)$ Cantor set} if
it has been constructed as in Section \ref{cantor-construct}, with the additional constraint that
all sets $S$ in $\Sigma$ are $\Lambda(p)$ sets contained in $[N-1]$ and obeying (\ref{lambda-p-0}) and (\ref{largeA}).
We will also say that the probability measure $\mu$ defined in Section \ref{cantor-construct} and supported
on $E_\infty$ is a \emph{$\Lambda(p)$ Cantor measure},
\end{definition}

Bourgain's theorem ensures that if $p\in (2,\infty)$ is given, then for all sufficiently large $N$ we can choose
$N_0=N_0(N)$ for which there exist $\Lambda(p)$ sets $S\subset[N-1]$ satisfying $|S|=N_0$ and obeying
the $\Lambda(p)$ assumptions (\ref{lambda-p-0}) and (\ref{largeA}) for some $c_0,c_1,C(p)$ independent of $N$. 
We will fix
these $c_0,c_1,C(p)$ throughout this paper, assume $N$ to be sufficiently large, and choose $N_0$ and $\Sigma$
accordingly. Note that $\alpha=(\log N_0)/(\log N)$ may depend slightly on $N$, but by (\ref{largeA}) we will
always have
\begin{equation}\label{good-alpha}
\frac{2}{p}+\frac{\log c_0}{\log N} \leq \alpha \leq \frac{2}{p}+\frac{\log c_1}{\log N} ,
\end{equation}
so that $\alpha$ can be as close to $2/p$ as we wish if $N$ is large enough.

\begin{theorem}\label{thm-main}
Let $2\leq p<\infty$. Then for any $\alpha'>2/p$, $\delta>0$, and for every  
$\Lambda(p)$ Cantor measure $\mu$ with $N$ sufficiently large depending on $p$ and $\alpha'$, we have

(i) $\mu$ is supported on a $\Lambda(p)$ Cantor set $E_\infty$ of Hausdorff dimension $\alpha < \alpha'$,

(ii) the maximal operator $\mathcal M$ given by (\ref{max-e1}) obeys the bound
\begin{equation}\label{max-e-main}
\|\calm f\|_p\leq C'_{N,\delta}\|f\|_{L^p_{\delta}(\RR)},\ \ f\in\cals,
\end{equation}
where $L^p_\delta(\RR)$ is the inhomogeneous Sobolev space with the norm
$$
\|f\|_{L^p_{\delta} }= \Big\| (1-\Delta)^{\delta/2} f\Big\|_p
= \Big\| \Big[ (1+|\xi|^2)^{\delta/2} \widehat{f}\, \Big]^\vee \Big\|_p.
$$

\end{theorem}

By interpolation with the trivial $L^\infty$ bound, (\ref{max-e-main}) implies the same bound with $p$ replaced by 
$q$ for any $q\in (p,\infty)$.


\subsection{Averaging estimates}

We briefly discuss the implications in terms of averaging estimates. Consider first a single-scale averaging operator
$f\to \mathcal{A}_1 f = f*\mu$ for a probability measure $\mu$ on $\RR$. By Young's inequality, we always have the trivial estimate
\begin{equation}\label{ave-e1}
\|\mathcal{A}_1 f\|_p \leq \|f\|_p,\ \ 1\leq p \leq \infty.
\end{equation}
If the measure $\mu$ satisfies a Fourier decay condition
\begin{equation}\label{fourier-decay}
\widehat{\mu}(\xi)\leq C_\beta (1+|\xi|)^{-\beta},
\end{equation}
we can improve this to an $L^2$-Sobolev estimate by writing
\begin{equation}\label{ave-e2}
\|\mathcal{A}_1 f\|_{L^2_\beta(\RR)} = \Big\| \Big[ (1+|\xi|^2)^{\beta/2} \widehat{\mu}(\xi) \widehat{f}(\xi)\, \Big]^\vee \Big\|_2 \lesssim \|f\|_2.
\end{equation}
It is well known (see e.g. \cite{Wolff}) that if $\mu$ is 
supported on a set of Hausdorff dimension $\alpha$, then
(\ref{fourier-decay}) can only hold for $\beta\leq \alpha/2$. We will say that a measure $\mu$ is a \textit{Salem measure} if 
it has optimal Fourier decay except possibly for the endpoint, i.e. (\ref{fourier-decay}) holds for all $\beta<\alpha/2$.
There are numerous constructions of such measures in the literature; within the framework of our construction of
$\Lambda(p)$ Cantor measures, we can ensure that $\mu$ is Salem by using the ``rotations mod $N$" technique 
of \cite{LP} as in \cite{Laba-Wang}. In that case, we get (\ref{ave-e2}) for all $\beta<\alpha/2$.

On the other hand, when the measure $\mu$ in Theorem \ref{thm-main} is self-similar, it is easy to check that
$\widehat{\mu}(N^j)\not\to 0$ as $j\to\infty$, so that an estimate of the form (\ref{fourier-decay}) cannot hold with any $\beta>0$. In fact,  For such measures, we cannot upgrade (\ref{ave-e1}) to a Sobolev
estimate (consider a sequence of functions with Fourier supports in $O(1)$ neighbourhoods of $N^j$).

Minor modifications of the proof of Theorem \ref{thm-main} yield the following estimates on 
averages of $\mathcal{A}_t f$ with respect to $t$. 

\begin{theorem}\label{thm-average}
Let $2\leq p<\infty$, and let $\mu$ be a   
$\Lambda(p)$ Cantor measure as in Theorem \ref{thm-main}. 
Then:

(i) for every $r$ with $p<r<\infty$, we have
\begin{equation}\label{lol}
 \Big\| \| \mathcal{A}_t f(x) \|_{L^r([1,2],dt)} \Big\|_{L^p(dx)}
\lesssim_N \|f\|_p,
\end{equation}
provided that $N$ is sufficiently large depending on $r$,

(ii) for $p=r$, we have the following Sobolev improvement for $\gamma<\alpha/2$:
\begin{equation}\label{lolsob}
\Big\| \Big[ (1+|\xi|^2)^{\gamma/2} \widehat{f}(\xi)\widehat{\mu}(t\xi)\, \Big]^\vee \Big\|_{L^p(\RR_x \times [1,2]_t)}
\lesssim_N \|f\|_p,
\end{equation}
provided that $N$ is sufficiently large depending on $\gamma$. We use $\ ^\vee$ to denote the inverse Fourier transform in $x$ only.

\end{theorem}

The estimates (\ref{lol}) and (\ref{lolsob}) hold for general $\Lambda(p)$ measures, including the self-similar case when no Fourier decay is available and we cannot do better than (\ref{ave-e1}) for a fixed $t$.


\subsection{A geometric corollary}

Let $X,Y\subset \RR$ be Lebesgue measurable. Suppose that for some choice of positive numbers $\{t(x)\}_{x\in Y}$ we have 
\begin{equation}\label{wtaf}
\bigcup_{x\in Y} (x+t(x)E_\infty) \subset X,
\end{equation}
where $E_\infty$ is a $\Lambda(p)$ Cantor set. 
Let $f=\one_X$. Then $\|f\|_p^p = |X|$, and $\calm f = 1$ on $Y$ so that $\|\calm f\|_p^p \geq |Y|$.
If we knew that $\calm$ is bounded on $L^p(\RR)$, it would follow that $|X|\gtrsim |Y|$; in particular, 
it would follow that if $Y$ has positive measure, then so does $X$,
We are not able to prove this, but we can prove the following weaker statement.

\begin{corollary}\label{WTAF}
Let $\delta>0$, and let $\mu$ be a $\Lambda(p)$ Cantor measure as in Theorem \ref{thm-main} for some 
$p\in[2,\infty)$, with $N$ large enough
depending on $p$ and $\delta$. Let $X,Y\subset\RR$ be sets such that (\ref{wtaf}) holds for some choice of $\{t(x)\}_{x\in Y}$.
Let $X_j$ and $Y_j$ denote the $N^{-j}$-neighbourhoods
of $X$ and $Y$. Then $|X_j|\gtrsim_N N^{-j\delta p} |Y_j|$, uniformly in $j$.
In particular, we have
$$
\overline{\dim_M}(X) \geq \overline{\dim_M} (Y) -\delta,
$$
where we use $\overline{\dim_M}$ to denote the upper Minkowski dimension of a set, and the same is true for
the  lower Minkowski dimension.
\end{corollary}

To see this, let $f_j=\one_{X_j} *\varphi_j$, where $\varphi_j(x)=N^j \varphi(N^j x)$ for a Schwartz function 
$\varphi$ such that $\varphi\geq 0$, $\varphi(x)\geq 1$ on $[-1,1]$, and $\widehat{\varphi}$ is supported in
$|\xi|\lesssim 1$. Then $\|f_j \|_{L^p_\delta}\lesssim N^{j\delta} |X_j|^{1/p}$, and 
$\calm f_j \gtrsim 1$ on
$Y_j$, so that $\|\calm f_j\|_p\gtrsim |Y_j|^{1/p}$. The conclusion follows from Theorem \ref{thm-main}.


\subsection{Literature overview}

Maximal and averaging operators associated with measures supported on lower-dimensional submanifolds of $\RR^d$ have been widely
studied in harmonic analysis. A fundamental prototype result in this area is the spherical maximal theorem, due to 
Stein \cite{stein-max} in dimensions $d\geq 3$ and Bourgain \cite{bourg-circles} for $d=2$, which asserts that 
the maximal operator associated with the Lebesgue measure on the sphere $S^{d-1}$ in $\RR^d$ is 
bounded on $L^p(\RR^d)$ for $p>{\frac{d}{d-1}}$. There is a large body of work on similar estimates under varying 
conditions on the dimensionality, smoothness and curvature of the underlying manifold, or on the Fourier decay 
of the measure $\mu$; see e.g. \cite{duo-book} or \cite{stein-ha} for a partial overview.

We mention a few prior results that allow fractal measures on $\RR^d$ with $d\geq 2$. 
A theorem of Rubio de Francia \cite{rdf} provides a maximal estimate for measures $\mu$
on $\RR^d$ that obey the Fourier decay condition (\ref{fourier-decay}) with $\beta>1/2$.
In particular, this allows 
fractal measures for which (\ref{fourier-decay}) holds. However, the result is void when $d=1$, since measures on $\RR$ that are not absolutely continuous 
can never satisfy (\ref{fourier-decay}) with $\beta>1/2$. In a different direction, Iosevich and Sawyer \cite{IS2003} proved a maximal estimate in the special case of spherically symmetric fractals in dimensions $d\geq 2$.
Iosevich, Krause, Sawyer, Taylor and Uriarte-Tuero \cite{IKSTU} studied a variant where the averages in
$\mathcal{A}_t$ are taken with respect to the spherical Lebesgue measure, but the $L^p$ norms of
$f$ and $\mathcal{A}_t f$ are evaluated with respect to fractal measures.

Relatively little is known about maximal estimates for fractal measures in dimension 1. 
The first such results were proved by the author and Pramanik in \cite{LP-diff}. 
Specifically, for any $0<\epsilon < \frac{1}{3}$,  there is a probability measure $\mu=\mu_\epsilon$ supported 
on a set $E_\infty \subset[1,2]$ of Hausdorff dimension $1 - \epsilon$ such that the associated maximal operator $\mathcal{M}$ is
bounded on $L^p(\RR)$ for $p > \frac{1 + \epsilon}{1 - \epsilon}$ (the best possible range would be
$p>1/(1-\epsilon)$). Furthermore, in the case corresponding to
$\epsilon=0$, there exists a probability measure $\mu$ supported 
on a set $E_\infty \subset[1,2]$, of Hausdorff dimension $1$ but Lebesgue measure 0, such that $\mathcal{M}$ is
bounded on $L^p(\RR)$ for all $p >1$.
This implies $L^p$ differentiation theorems with the same range of $p$ (answering a question of Aversa and Preiss).
Results on $L^p\to L^q$ boundedness of appropriately modified maximal operators are also obtained.
The construction in \cite{LP-diff} is probabilistic and relies on ``correlation conditions" (essentially, estimates on the size of intersections of two or more rescaled and translated copies of the support of $\mu$). 
It does not produce explicit examples or allow self-similar sets.

Shmerkin and Suomala \cite{SS-prep} have told me that they were able to improve this as follows:
for $k\in \{2,3,\dots\}$, and for $p_0$ with the dual exponent $p'_0=k$, they construct Cantor measures $\mu$ of dimension $\alpha= \frac{1}{p_0} =1-\frac{1}{k}$ such that $\mathcal{M}$ is
bounded on $L^p$ for the optimal range $p>p_0$. Their proof follows the general scheme of \cite{LP-diff}, but
with improved correlation conditions obtained via the methods of \cite{shmerkin-suomala2014}, \cite{shmerkin-suomala2016}.

It turns out to be very difficult to decide whether specific fractal measures can differentiate $L^p(\RR)$ for sufficiently large but finite $p$.
Math\'e (unpublished, see \cite{Keleti-diff}) has reportedly constructed explicit fractal measures on $\RR$ that cannot differentiate
$L^p(\RR)$ for any $p<\infty$. The problem remains open for self-similar measures, including the middle-third Cantor measure (this question was already raised by Aversa and Preiss in the 1990s; see \cite{LP-diff} for a more thorough discussion of the relevant history).
Hochman \cite{Hochman-diff} proved using entropy methods from \cite{H2014} that if $X\subset\RR$ contains a scaled copy of a Cantor set
$K$ centered at every point of a set $Y$ of positive Hausdorff dimension, then $\dim_H(X)>\dim_H(K)$; however, 
the proof of differentiation would require
a similar estimate with $\dim_H(K)$ replaced by $\dim_H(Y)$ which can be much larger. Our Corollary \ref{WTAF} is a partial result in that direction.

Our present approach via decoupling is not sufficient to yield $L^p$ boundedness of $\mathcal{M}$ for any 
$p<\infty$. It is likely that this will require an additional combinatorial argument; we hope to address this in a 
future paper. On the other hand, Theorem \ref{thm-main} extends the study of maximal operators for Cantor sets on the line in several directions that were not covered in \cite{LP-diff}, \cite{SS-prep}. We can allow $\alpha'$, therefore $\alpha$,
to be arbitrarily small (in \cite{LP-diff}, we require $\alpha>2/3$; Shmerkin and Suomala require $\alpha\geq 1/2$).
Our construction of $\Lambda(p)$ Cantor sets allows self-similar measures, with
$S=S_{j,a}$ the same for all $j$ and $a$. Furthermore, explicit constructions of $\Lambda(p)$ sets are 
available in some cases (e.g. Sidon sets for $p=4$, see \cite{singer}, \cite{bose}, \cite{ruzsa}), hence we can give explicit examples of measures for which the theorem
holds.  This also shows that Theorem \ref{thm-main} can hold for measures 
on $\RR$ without
Fourier decay\footnote{In higher dimensions, a related but different phenomenon arises in the work of  
Keleti, Nagy and Shmerkin \cite{KNS}, Thornton \cite{thornton}, and Olivo and Shmerkin \cite{OS} on packing theorems and maximal operators
associated with cube skeletons.}.
(In \cite{LP-diff}, the ``correlation condition" imposed on our measures forced them to obey
(\ref{fourier-decay}) with some $\beta>0$.)

An unpleasant feature of the problem is that there does not seem to be an easy way to use 
(\ref{fourier-decay}) to obtain further improvements in Theorems \ref{thm-main} and \ref{thm-average}, 
even when $\mu$ is Salem. This is in contrast to papers such as \cite{MSS} or \cite{pra-seeger}, where
both decoupling (or square function estimates) and Fourier decay play a role. One issue is that Fourier-analytic proofs of maximal theorems usually require good estimates on the derivatives of $\widehat{\mu}$,
which are not available in our case.


\subsection{Outline of proof}\label{outline}


We will follow a Fourier-analytic approach, developed in \cite{MSS} and then adapted in \cite{wolff-smoothing},
\cite{pra-seeger} to use decoupling instead of square function estimates. 
Let $$
F_\gamma f(x,t)=\langle D_t \rangle^\gamma \big( \rho(t)\mathcal{A}_tf(x)\big),\ \ f\in\mathcal{S}, 
$$
where
$\rho$ is a smoothed out characteristic function of $[1,2]$, $D_t = \frac{1}{2\pi i} \frac{\partial}{\partial t}$ and $\langle u \rangle = (1+|u|^2)^{1/2}$.
Suppose that we could prove that
\begin{equation}\label{main-estimate-fake}
\|F_\gamma f\|_{L^p(dxdt)} \lesssim \|f\|_{L^p(dx)},\ \ f\in\mathcal{S},
\end{equation}
for some $\gamma>1/p$. Then by the Sobolev embedding theorem, we would have
\begin{align*}
\sup_{t} |\mathcal{A}_t f(x)|
\lesssim  \|F_\gamma f(x,\cdot) \|_{L^p(dt)} ;
\end{align*}
taking the $L^p$ norms in $x$ would then yield a maximal estimate. We will not be able to actually
prove (\ref{main-estimate-fake}), so instead we proceed as follows to get a weaker estimate.

By a standard reduction (see Section \ref{one-scale-section}), it suffices to consider 
the single-scale maximal operator $\tilde{\mathcal{M}}$ with the range of $t$ restricted to $[N^{-1},1]$.
We will be seeking bounds of the form
\begin{equation}\label{fake-2}
\|\tilde\calm f\|_p\leq C_N N^{j\beta} \|f\|_p,\ \ j\geq j_0,
\end{equation}
for some $\beta\in\RR$ and
for all $f$ with $\widehat{f}$ supported in $|\xi|\sim N^j$. Adding the appropriate cut-offs and then 
applying the Sobolev embedding argument, we reduce the problem to estimating
a Fourier multiplier operator $F_j$ given by 
$$
\widehat{F_{j}f} (\xi,s) =  \tilde{m}_{j}(\xi,s) \widehat{f} (\xi),\ \ f\in\cals,
$$
with the multiplier $m_j$ supported (up to small errors) on a neighbourhood of the Cantor bush
$\mathcal{K}_j=\bigcup_{a\in A_j} \mathcal{K}_{j,a}$, where
\begin{align*}
\mathcal{K}_{j,a}=  \left\{ (\xi,s)\in\RR^2:\ \   
|\xi|\sim N^j,\ \ |\xi a -s| \leq 1 \right\}.
\end{align*}

Let $F_j f=  \sum_{a\in A_j} F_{j,a}f $, where (again, up to small errors) $F_{j,a}f$ is Fourier supported on 
$\mathcal{K}_{j,a}$.
The main ingredient of the proof is the decoupling estimate
\begin{equation}\label{decoupling-early}
\|F_{j}f \|_p  \lesssim N^{j\epsilon} \Big( \sum_{a\in A_{j}} \|F_{j,a} f \|_p^2 \Big)^{1/2}
\end{equation}
for some small $\epsilon>0$. 
For each individual $a\in A_j$, we have the estimate
$$
\| F_{j,a}f \|_p \lesssim N^{j(\gamma-\alpha)}  \|f\|_p,
$$
which can be proved by writing out $F_{j,a}$ in its integral operator form and using Young's inequality.
Plugging this into (\ref{decoupling-early}), and summing over $a\in A_j$ with $|A_j|=N^{j\alpha}$, we get
\begin{equation}\label{almost-conclude}
\|F_jf\|_p \lesssim  N^{j \epsilon}
N^{j(\gamma - \alpha/2 )}  \|f\|_p.
\end{equation}
This implies (\ref{fake-2}) with $\beta= \gamma - \frac{\alpha}{2}+\epsilon$. In order to apply Sobolev's embedding theorem, we must have $\gamma>1/p$, and recall from (\ref{good-alpha}) that $1/p$ is very close to $\alpha/2$.
Thus we will not get (\ref{fake-2}) with $\beta>0$ 
(which would be needed in order to prove that $\tilde{\mathcal{M}}$, and therefore $\calm$, is bounded on $L^p$), but we will be able to arrange for $\beta>0$ to be arbitrarily small by taking sufficiently large $N$, which leads
to Theorem \ref{thm-main}.

The decoupling estimate used in (\ref{decoupling-early}) is proved in Section \ref{sec-decoupling}. The argument is
similar to that of  {\L}aba and Wang \cite{Laba-Wang} for $\Lambda(p)$ Cantor sets on the line, and in fact uses
the single-step decoupling inequality from \cite[Lemma 5]{Laba-Wang} as a basic building block. 
The proof in \cite{Laba-Wang} is, in turn, based
on iterating a continuous variant of Bourgain's $\Lambda(p)$ estimate in Theorem \ref{bourgain-thm},
and on the decoupling techniques from the work of Bourgain and Demeter \cite{BD2015}, \cite{BD-expo}. 
The additional geometric observation needed to prove a similar estimate for functions on $\RR^2$ with
Fourier transforms supported in the Cantor bush is the following: at each step of the iterative construction of the Cantor bush, the $k+1$-level branches 
contained in a single $k$-th level branch 
are close to parallel when restricted to $\xi$-intervals of length $N^{j-O(1)}$. This allows us to apply
the ``parallel decoupling" argument (cf. \cite[Section 8]{BD2015}) to pass from one-dimensional 
Cantor sets to a two-dimensional bush.

It is easy to see that, in general, (\ref{decoupling-early}) cannot hold with the exponent $p$  replaced
by $q$ with $q>p$. Indeed, if that were possible, then we could just consider functions whose Fourier transform is supported on
a thin horizontal slice $\{\xi_0\leq  \xi \leq \xi_0+1\}$ of the Cantor bush. Then the estimate (\ref{decoupling-early}) becomes essentially
one-dimensional and any improvement in the exponent would have to correspond to a similar improvement in Bourgain's 
$\Lambda(p)$ theorem, which is known to be impossible. 
We also note that Demeter \cite{demeter-cantor} has proved a decoupling estimate for Cantor sets on 
a parabola; while there is at least a nominal similarity to this paper, his result is based on the curvature of the parabola and
does not apply in our setting.


\section{Initial reductions}


\subsection{Notation}

We write $[N]=\{0,1,\dots,N-1\}$.  For $d=1,2$, we use $|\cdot |$ to denote the Euclidean norm of a vector in $\RR^d$, the cardinality of a finite set, or the $d$-dimensional Lebesgue measure of a subset of $\RR^d$, depending on the context. We will also write $B(x,r)=\{y\in\RR^d:\ |x-y|\leq r\}$.
If $b\in\RR^d$, $c\in\RR$, and $B_1,B_2\subset  \RR^d$, we write
$b+B_1=\{b+b_1:\ b_1\in B_1\}$, $cB_1=\{cb_1:\ b_1\in B_1\}$, and $B_1+B_2=\{b_1+b_2:\ b_1\in B_1,\ b_2\in B_2\}$. .

We use $X \lesssim Y$ to say that $X \leq C Y$ for some constant $C>0$, and $X \sim Y$ to say that $X \lesssim Y$ and $X \gtrsim Y$. The constants such as $C,C'$, etc. and the implicit constants in $\lesssim$ may change from line to line, and may depend on $d$ and $p$, but are independent of variables or parameters such as $x,R, j, k$.
Whenever a constant depends on $N$, we will indicate this explicitly by writing $C_N$, $C(N)$, $X\lesssim_N Y$,
etc; all other constants will be independent of $N$. 

A word on how the constants are organized: in our main decoupling inequality (Proposition \ref{multidecoupling}), we lose a factor of
the form $C^j$ with $C$ independent of $j$. We then want to argue that, given $\epsilon>0$, this can be dominated by $N^{j\epsilon}$, provided that $N$ was chosen large enough depending on $\epsilon$. In order for this to work, it is crucial that the constant $C$ and all constants leading up to it be independent of $N$ as well. All other parts of the proof are
non-iterative and the dependence of the constants there on $N$ is harmless.

For a function $f:\RR\to\CC$, we define its Fourier transform 
$$
\widehat{f}(\xi) =  \int e^{-2\pi i x \xi} f(x)  dx, \qquad  \xi \in \RR,
$$
and similarly for a measure $\mu$ on $\RR$, 
$$
\widehat{\mu}(\xi) = \int e^{-2\pi i x \xi}   d\mu(x) \qquad  \xi \in \RR.
$$
For functions $f:\RR_{x}\times \RR_t\to\CC$, we reserve $\xi$ and $s$ to denote the Fourier variables dual to 
$x$ and $t$ respectively, so that
$\widehat{f}(\xi,s) =  \iint e^{-2\pi i (x  \xi+ts)} f(x,t)  dxdt$. 
We will also sometimes use $\mathcal{F}$ for the Fourier transform, so that $\mathcal{F}f =\widehat{f}$.
If the Fourier transform of a function $f(x,t)$ is taken only in one variable, we will indicate this using  
subscripts, e.g.,
$\mathcal{F}_{x\to\xi}f$. 
We will use the notation
$$
D_x = \frac{1}{2\pi i} \frac{\partial}{\partial x}
$$
 so that $D_x e^{2\pi i x \xi} = \xi e^{2\pi i x  \xi}$, and similarly for the $t$ variable. 
We will also write 
$\langle u \rangle = (1+|u|^2)^{1/2}$. If $p\in[1,\infty]$, we use $p'$ to denote the dual exponent
defined via $\frac{1}{p}+\frac{1}{p'}=1$.

Throughout the rest of this paper, $\mu$ will be a $\Lambda(p)$ Cantor measure as in Theorem \ref{thm-main}.
We also need additional notation associated with Cantor sets. In the introduction, we defined
$A_{k+1,a}=a+N^{-(k+1)}A$ for $a\in A_k$, so that $A_{k+1}=\bigcup_{a\in A_k} A_{k+1,a}$.
Let  
$$
E_{k+1,a}:=A_{k+1,a}+[0,N^{-(k+1)}] = E_{k+1} \cap [a,a+N^{-k}]
$$
so that  $E_{k+1}=\bigcup_{a\in A_k} E_{k+1,a}$. 
We will also use the decomposition $\mu=\sum_{a\in A_k} \mu_{k,a}$, where
$$
\mu_{k,a} = \mu \big|_{a+[0,N^{-k}]}.
$$
In the self-similar case, $\mu_{k,a}$ is a similar copy of $\mu$, rescaled to $a+[0,N^{-k}]$ and 
with total mass $N_0^{-k}=N^{-k\alpha}$.


\subsection{Reduction to a single scale}\label{one-scale-section}

\begin{lemma}\label{lemma-single-scale}
Define the restricted maximal operator 
\begin{equation} \label{max-e103}
\tilde\calm f(x) :=  \sup_{N^{-1}\leq t \leq 1} \left| \int   f(x -ty)  d\mu(y)\right|  
 =\sup_{N^{-1}\leq t \leq 1} |\cala_t f(x)|,\ \ f\in\cals,
\end{equation}
with $\cala_tf := \int f(x-ty)\, d\mu(y)$ as before. Let $2\leq p <\infty$.
Suppose that for some $j_0\in\NN$ we have the estimate
\begin{equation}\label{e-m-restricted}
\|\tilde\calm f\|_p\leq C_N N^{j\beta} \|f\|_p,\ \ j\geq j_0,
\end{equation}
for all $f\in\cals$ with $\supp \widehat{f}\subset \{N^j\leq|\xi| \leq 2N^{j+1}\}$, with the constant $C_N$
independent of $j$. Then the full maximal operator $\calm$ defined in (\ref{max-e1}) obeys
\begin{equation}\label{max-ee1}
\|\calm f\|_p\leq C'_N\|f\|_p \ \ \hbox{if}\ \beta<0,
\end{equation}
\begin{equation}\label{max-ee2}
\|\calm f\|_p\leq C'_{N,\epsilon}\|f\|_{L^p_{\beta+\epsilon}} \ \ \hbox{if}\ \beta>0.
\end{equation}
\end{lemma}

\begin{proof}
The argument here is well known (see \cite{bourg-circles}), but since it is short and we need to keep
track of the scaling, we include the proof for completeness. Our presentation follows \cite{schlag-thesis}, with the scaling factor 2 replaced by $N$.

It suffices to prove that the bounds (\ref{max-ee1}), (\ref{max-ee2}) hold with $\calm$ replaced by
$$
{\mathcal M_R} f(x) :=  \sup_{0<t<R} \left| \int f(x - ty) d\mu(y)\right| ,
$$
and with constants independent of $R$. By scaling, it suffices to consider $R=1$.

Let $\phi \in C_c^\infty(B(0,2))$ be a function such that $0\leq\phi \leq 1$ and $\phi  (\xi)=1$ for $|\xi|\leq 1$.
Define $\phi_0(\xi)=\phi(N^{-1}\xi)$
and $\phi_j(\xi)=\phi(N^{-j-1 }\xi) - \phi(N^{-j}\xi)$ for $j\in\NN$. Then $\sum_{j=0}^\infty \phi_j\equiv 1$
and $\phi_j$ is supported in the region $N^{j}\leq |\xi|\leq 2N^{j+1}$ for $j>0$. Define $f_j$ via
$$
\widehat{f_j} =\phi_j \widehat{f}
$$
so that $f = \sum_{j=0}^\infty f_j$. Then
\begin{align*}
{\mathcal M_1} f(x) 
&= \sup_{k\geq 0}\sup_{N^{-k-1}\leq t\leq N^{-k}} |\cala_t f(x)|
\\
&\leq  \sup_{k\geq 0}\sup_{N^{-k-1}\leq t\leq N^{-k}} \left| \cala_t \Big[ \sum_{j < j_0+ k} f_j\Big] (x)   \right|
\\
&\ \ \ + \sup_{k\geq 0}\sup_{N^{-k-1}\leq t\leq N^{-k}} \left| \cala_t \Big[ \sum_{j \geq j_0+ k} f_j\Big] (x)   \right|
\\
&=:I_1(x)+I_2 (x).
\end{align*}
The $I_1$ part is dominated by a constant (depending on $N,j_0$) multiple of the Hardy-Littlewood maximal operator, therefore bounded on all $L^p$ with
$p>2$. To estimate $I_2$, we use (\ref{e-m-restricted}) and scaling. Let $f_{j,k}(x)=N^{-k/p}f_j(N^{-k}x)$, then
$\|f_{j,k}\|_p=\|f_j\|_p$ and 
$\widehat{f_{j,k}}$ is supported in $N^{j-k}\leq |\xi|\leq 2N^{j-k+1}$. 
We have 
$$\cala_t f_j(x) = N^{k/p} \big[ \cala_{N^kt} f_{j,k}\big] (N^kx),$$
 so that
\begin{align*}
I_2(x)& \leq 
\sup_{k\geq 0} \sum_{j\geq j_0+ k} \sup_{N^{-k-1}\leq t\leq N^{-k}} \left| \cala_t  f_j (x) \right| 
\\
&\leq  \sup_{k\geq 0} \sum_{j \geq j_0+ k} N^{k/p} \tilde\calm  f_{j,k} (N^k x) 
\\
&\leq \left( \sum_{k\geq 0} \left[ \sum_{j\geq j_0+ k} N^{k/p} \tilde\calm  f_{j,k} (N^k x) \right]^p \right)^{1/p}.
\end{align*}
It follows that
\begin{align*}
\|I_2\|_p &\leq \left( \int \sum_{k\geq 0} \left[ \sum_{j\geq j_0+ k} N^{k/p} \tilde\calm  f_{j,k} (N^k x) \right]^p dx \right)^{1/p}
\\
&= \left( \sum_{k\geq 0} \left\|  \sum_{j\geq j_0+ k} N^{k/p} \tilde\calm  f_{j,k} (N^k \cdot ) \right\|_p^p \right)^{1/p}
\\
&\leq \left( \sum_{k\geq 0}  \left[ \sum_{j\geq j_0+k} \| \tilde\calm  f_{j,k} \|_p \right]^p \right)^{1/p}
\\
&\lesssim_N 
\left( \sum_{k\geq 0}  \left[ \sum_{j\geq j_0+k} N^{(j-k)\beta}\|  f_{j} \|_p \right]^p \right)^{1/p},
\end{align*}
where at the last step we used (\ref{e-m-restricted}) and that $\|f_{j,k}\|_p=\|f_j\|_p$.
If $\beta<0$, we use discrete Young's inequality and then Littlewood-Paley to estimate
\begin{align*}
\|I_2\|_p  \lesssim_N \Big( \sum_j \|f_j\|_p^p \Big)^{1/p}
\lesssim \Big\| \Big( \sum_j |f_j|^2 \Big)^{1/2} \Big\|_p
\lesssim \|f\|_p.
\end{align*}
If on the other hand $\beta>0$, we have instead 
\begin{align*}
\|I_2\|_p  &\lesssim_N 
\left( \sum_{k\geq 0}  N^{-k\beta p} \left[ \sum_{j\geq j_0+k} N^{-j\epsilon} N^{j(\beta+\epsilon)} \|  f_{j} \|_p \right]^p \right)^{1/p}
\\
&\lesssim_N 
\left( \sum_{k\geq 0}  N^{-k\beta p} \left[ \sum_{j\geq j_0+k} N^{-j\epsilon}  \|  \langle D_x \rangle^{\beta+\epsilon}f \|_p \right]^p \right)^{1/p}
\\
&\lesssim_{N,\epsilon}  \|f\|_{L^p_{\beta+\epsilon}}.
\end{align*}

\end{proof}


\subsection{A multiplier problem}
\label{sobolev}

Following \cite{MSS} (see also \cite{pra-seeger}), we perform a further reduction as follows. For $\gamma>0$, 
we define the operator $F_{\gamma}$, mapping functions $f\in\cals(\RR^d_x)$ to Schwartz functions on
$\RR^d_x\times \RR_t$: 
$$
F_\gamma f(x,t)=\langle D_t \rangle^\gamma \big( \rho(t)\cala_tf(x)\big),\ \
$$
where $\rho\in C_c^\infty (\frac{1}{2N}, 2)$ is a fixed function such that $\rho\geq 0$ and $\rho\equiv 1$ on $[\frac{1}{N},1]$.
Recall that $D_t = \frac{1}{2\pi i} \frac{\partial}{\partial t}$ and $\langle u \rangle = (1+|u|^2)^{1/2}$, so that
for a function $h(t)$ we have
$$
\mathcal{F} (\langle D_t \rangle^\gamma h) (s) = \sgamma\, \widehat{h}(s).
$$

Suppose that we can prove that for some function $f\in \cals$,
\begin{equation}\label{main-estimate}
\|F_\gamma f\|_{L^p(dxdt)} \lesssim K \|f\|_{L^p(dx)},
\end{equation}
for some $\gamma$ such that $\gamma p>1$. Then by the Sobolev embedding theorem, we have 
\begin{align*}
\tilde\calm f(x) \lesssim   \sup_{N^{-1}\leq t \leq 1} |\rho(t) \cala_t f(x)|
\lesssim  \|F_\gamma f(x,\cdot) \|_{L^p(dt)} ,
\end{align*}
so that
\begin{equation}\label{main-estimate-b}
\|\tilde\calm f\|_{L^p(dx)} \lesssim   \|Ff \|_{L^p(dxdt)} \lesssim K \|f\|_{L^p(dx)}.
\end{equation}
Our strategy will be to prove (\ref{main-estimate}) (therefore (\ref{main-estimate-b})) 
for all $f\in\cals$ such that $\supp \widehat{f}\subset \{N^j\leq|\xi| \leq 2N^{j+1}\}$, with $K\lesssim_{N} N^{j\beta}$
uniformly in $j$,
then use Lemma \ref{lemma-single-scale} to pass to the unrestricted maximal operator. 

We first set up the appropriate band-limited operators.
Let $\phi$ and $\phi_j$ be the functions defined at the beginning of the proof of Lemma 
\ref{lemma-single-scale}, and define $\sigma_j$ via
$$
\widehat{\sigma_j} =\phi_j \widehat{\mu}
$$
so that $\sigma_j\in\cals$ and $\widehat{\mu} = \sum_{j=0}^\infty \widehat{\sigma_j} $.
Let
$$
F_{\gamma,j} f(x,t)=\langle D_t \rangle^\gamma \left( \rho(t) \int \widehat{f} (\xi) \widehat{\sigma_j}(t\xi) e^{2\pi i x \xi} d\xi
\right).
$$

\begin{lemma}\label{lemma-Fjsuffices}
Let $2\leq p<\infty$ and $\gamma>1/p$. With $F_{\gamma,j}$ as above, suppose that we have the estimate
\begin{equation}\label{main-est-j}
\|F_{\gamma,j} f\|_{L^p(dxdt)} \lesssim_N N^{j\beta} \|f\|_{L^p(dx)},\ \ f\in\cals.
\end{equation}
Then $\calm$ obeys the conclusions (\ref{max-ee1}) or (\ref{max-ee2}) of Lemma \ref{lemma-single-scale},
depending on the sign of $\beta$.
\end{lemma}

\begin{proof}
We prove below in Lemma \ref{mult-lemma1} that $F_{\gamma,]}$ is a Fourier multiplier operator on $\RR^2_{x,t}$
with a multiplier supported in $\half N^j\leq |\xi| \leq 4N^{j+2}$. It follows that for functions $f$ with $\supp \widehat{f}\subset \{N^j\leq|\xi| \leq 2N^{j+1}\}$ with $j\geq 3$, we have 
$$
F_\gamma f = \sum_{k=j-2}^{j+1}  F_{\gamma,k}f.
$$
Therefore, if (\ref{main-est-j}) holds, then so do (\ref{main-estimate}) and (by the above discussion) 
(\ref{main-estimate-b}) with $K=C_N N^{j\beta}$. Hence the assumption (\ref{e-m-restricted}) holds with
$j_0=3$, and the conclusion follows from
Lemma \ref{lemma-single-scale}.
\end{proof}

In the sequel, $\gamma>0$ will be fixed and we will
omit it from notation, writing $F_{\gamma,j}=F_j$. 

\begin{lemma}\label{mult-lemma1}
We have the Fourier multiplier representation
\begin{equation}\label{fajfa-1}
\widehat{F_jf} (\xi,s) =  \tilde{m}_j(\xi,s) \widehat{f} (\xi),\ \ f\in\cals,
\end{equation}
where $\tilde{m}_j(\xi,s) $ is a Schwartz function in $2$ variables, given by
\begin{align*}
\tilde{m}_j(\xi,s) &=  \sgamma \int \sigma_j (y) \widehat{\rho} (\xi  y+ s) dy .
\end{align*}
and supported in $\half N^j\leq |\xi| \leq 4N^{j+2}$.
\end{lemma}

\begin{proof}
Let
$$\cala_{t,j} f := \int f(x-ty)\sigma_j(y)\,dy = \int \widehat{f}(\xi) \widehat{\sigma_j}(t\xi) e^{2\pi i x \xi} d\xi .$$
Then $\rho(t) \cala_{t,j} f\in\cals_{x,t}$, therefore so does $F_jf$. Taking the partial Fourier transform in $t$, we get
\begin{align*}
\mathcal{F}_{t\to s}\left( F_jf \right)
(x,s)
&= \sgamma  \mathcal{F}_{t\to s} \left( \rho(t) \cala_{t,j}f\right) (x,s)
\\
&= \sgamma \int e^{-2\pi i ts} \rho(t) \int \widehat{f}(\xi) \widehat{\sigma_j}(t\xi) e^{2\pi i x \xi} d\xi dt .
\end{align*}
Interchanging the order of integration, we get that
\begin{equation}\label{oops-a}
\mathcal{F}_{t\to s}\left( F_jf \right)
(x,s)
= \int e^{2\pi i x \xi}  \widehat{f}(\xi)  \tilde{m}_j (\xi,s) d\xi ,
\end{equation}
where
$$
\tilde{m}_j(\xi,s) = \sgamma \int e^{-2\pi i ts} \rho(t) \widehat{\sigma_j }(t\xi)  dt.
$$
For $t\in\supp \rho \subset [\frac{1}{2N},2]$,
 $\widehat{\sigma_j} (t\xi)$ as a function of $\xi$ is supported in $[t^{-1}N^{j},2t^{-1}N^{j+1}]
\subset [\frac{1}{2}N^{j}, 4N^{j+2}] $. Therefore $\tilde{m}_j$ 
is a Schwartz function supported in $\half N^j\leq |\xi| \leq 4N^{j+2}$.

Next, we rewrite $\tilde{m}_j$ as
\begin{align*}
\tilde{m}_j(\xi,s) 
&=  \sgamma \int e^{-2\pi i ts} \rho(t) \int \sigma_j(y) e^{-2\pi i t \xi  y} dy  dt
\\
&=  \sgamma \int \sigma_j(y) \Big[  \int \rho(t) e^{-2\pi i t (\xi  y + s) } dt \Big] dy
\\
&= \sgamma \int \sigma_j (y) \widehat{\rho} (\xi  y +s) dy,
\end{align*}
as claimed.
Finally, taking the Fourier transform in $x$ in (\ref{oops-a}) proves (\ref{fajfa-1}) and completes the proof of the lemma.
\end{proof}


\section{Localization estimates}\label{localization}


Recall that $\mu=\sum_{a\in A_j} \mu_{j,a}$, where
$\mu_{j,a} = \mu \big|_{a+[0,N^{-j}]}$.
Accordingly, let
$$\widehat{\sigma_{j,a}} =\phi_j \widehat{\mu_{j,a}}.$$
Then $F_j f = \sum_{a\in A_j} F_{j,a}f$, where
\begin{equation}\label{local-e2}
\begin{split}
&\widehat{F_{j,a}f} (\xi,s) = \tilde{m}_{j,a}(\xi,s) \widehat{f} (\xi),
\\
&\tilde{m}_{j,a}=\sgamma \int \sigma_{j,a} (y) \, \widehat{\rho} (\xi  y+s) dy .
\end{split}
\end{equation}

Recalling the definition of $\phi_j$, we see that $\sigma_{j,a} = \mu_{j,a} * N^{j} \psi (N^j\cdot)$, where
$\psi := \mathcal{F}^{-1} (\phi(N^{-1}\xi) - \phi (\xi))$ is a fixed Schwartz function. Moreover, $\mu_{j,a}$ is
a rescaling to the interval $a+[0,N^{-j}]$ of a Cantor measure of the same type as $\mu$, with the same $N$ and $t$, 
and with total mass $N^{-j\alpha}$.
It follows that
$$
\sigma_{j,a}(y) = N^{j(1-\alpha)} \tsigma (N^{j}(y-a)),
$$ 
where $\tsigma\in\cals$ have all Schwartz seminorms bounded uniformly in $a$ and $j$. 

We now fix an $a\in A_{j}$. For that $a$, we define new coordinates $(\xi, \tau_a)$ in the Fourier space:
$$
(\xi, \tau_a )  := (\xi, s+ a \xi ).
$$
In the rest of this section, $a$ will be fixed and we will supress the dependence of $\tau_a$ on $a$, writing 
$\tau=\tau_a$,
Then
\begin{align*}
\tilde{m}_{j,a} (\xi,s) &= N^{-j\alpha} \langle a \xi - \tau \rangle^\gamma  
\int  N^{jd} \tsigma (N^j (y-a)) \,  \widehat{\rho} (\xi  y- a \xi + \tau) dy
\\
&=  N^{-j\alpha} \langle a \xi - \tau \rangle^\gamma
 \int N^{j} \tsigma (N^j y)  \, \widehat{\rho} (\xi  y + \tau) dy
 \\
&=  N^{-j\alpha} \langle a \xi - \tau \rangle^\gamma
\int  \tsigma( y)   \,\widehat{\rho} (N^{-j} \xi  y+ \tau) dy.
\end{align*}
Define
\begin{equation}\label{e-ma}
m_{j,a} (\xi,\tau) := N^{j(\alpha-\gamma)} \tilde{m}_{j,a} (\xi,s)
= \frac{ \langle a  \xi -\tau  \rangle^{\gamma} }{N^{j\gamma}}  \lambda_{j,a} (\xi,\tau),
\end{equation}
where 
\begin{align*}
\lambda_{j,a}(\xi,\tau)
&= \int  \tsigma( y)   \,\widehat{\rho} (N^{-j} \xi  y+ \tau) dy.
\end{align*}
In addition to changing variables, we also normalized the multipliers to simplify the forthcoming calculations.

We also note the following representations of $F_{j,a}$ as a Fourier integral operator.
\begin{lemma}
We have
\begin{equation}\label{fajfa-2}
F_{j,a} (x,t) =  N^{j(\gamma-\alpha)} \int  f(y)  {K}_{j,a}(x-y,t) dy
\end{equation}
where
\begin{equation}\label{fajfa-3}
K_{j,a}(x ,t) =\iint e^{2\pi i ((x-ta) \xi + t\tau)} m_{j,a} (\xi,\tau ) d\xi d\tau.
\end{equation}
\end{lemma}

\begin{proof}
Taking the inverse Fourier transform of the first equation in (\ref{local-e2}) in both $x$ and $t$, we get that
$F_{j,a} (x,t) = \int  f(y)  \tilde{K}_{j,a}(x-y,t) dy$, where
$$
\tilde{K}_{j,a}(x,t) =\iint e^{2\pi i (x  \xi + ts)} \tilde{m}_{j,a} (\xi,s) d\xi ds.
$$
Substituting $\tilde{m}_{j,a} (\xi,s)= N^{j(\gamma-\alpha)}m_{j,a} (\xi,\tau)$ and  
changing variables $(\xi,s)\to (\xi,\tau)$ in the integral defining $\tilde{K}_{j,a}$, with Jacobian $\Big|\frac{\partial(\xi,\tau)}{\partial(\xi,s)}\Big|=1$, we get (\ref{fajfa-3}).
\end{proof}

\begin{proposition}\label{prop-k-decay}
For any $M\in\NN$, we have
\begin{align}\label{k-decay}
\big|K_{j,a}(x,t)\big| \leq C_{M,N} N^j \min & \left[ \Big(1+ N^{2j}|x- ta|^2 \Big)^{-M}, \right.
\\
& \left. \Big(1+ |t|^2 \Big)^{-M} \right] \label{k-decay-t}
\end{align}
with $C_{M,N}$ independent of $j$ and $a$. 

\end{proposition}

\begin{proof}
We first prove the bound in (\ref{k-decay}).
Let $u=x-ta$.  Let also $L_M=(1+ N^{2j}|u|^2 )^{-M} (1+N^{2j} D_\xi^2)^M $, so that $L_M e^{2\pi i u\xi} = e^{2\pi i u\xi}$. 
Integrating by parts in $\xi$ (with boundary terms 0, since $m_{j,a}\in\cals$), we get
\begin{align*}
(1+ N^{2j}|u|^2 )^{M} K_{j,a}(x,t)
& =(1+ N^{2j}|u|^2 )^{M} \iint e^{2\pi i t\tau} (L_M e^{2\pi i u\xi} ) m_{j,a} (\xi,\tau ) d\xi d\tau
\\
& =  (1+ N^{2j}|u|^2 )^{M} \iint e^{2\pi i (t\tau  +u \xi) } ) ( L_M m_{j,a}) (\xi,\tau ) d\xi d\tau
\\
&= \iint e^{2\pi i (t\tau +u \xi) } ) \Big( ( 1+N^{2j} D_\xi^2)^M m_{j,a}\Big) (\xi,\tau ) d\xi d\tau.
\end{align*}
Hence
$$
\big|K_{j,a}(x,t)\big| \leq (1+ N^{2j}|u|^2 )^{-M} I_{M,N}(j),
$$
where
\begin{align*}
I_{M,N} (j) &= \iint \Big|  \Big( ( 1+N^{2j} D_\xi^2)^M m_{j,a}\Big) (\xi,\tau ) \Big| d\xi d\tau
\\
& \leq  \sum_{0\leq n_1+n_2\leq 2M} 
\iint  (N^j D_\xi)^{n_1}   \frac{ \langle a\xi - \tau \rangle^{\gamma} }{N^{j\gamma}}  
\cdot (N^j D_\xi)^{n_2} \lambda(\xi,\tau)
d\xi d\tau.
\end{align*}
We need to prove that $I_{M,N}\lesssim_{M,N} N^j$. 
Recall that $m_{j,a}$ is supported in $\half N^j \leq |\xi| \leq 4N^{j+2}$, so that the integration in $I_{M,N}$ is restricted
to the same region. Hence it suffices to prove that the
integrands are bounded by $C_{M,N} (1+|\tau|)^{-M}$ with constants independent of $j$. This follows from
the estimates (\ref{garbage-e2}) and (\ref{garbage-e1}) below.

\begin{itemize}
\item We have $\lambda_{j,a}(\xi,\tau)
= \int  \tsigma( y)   \, \widehat{\rho} (N^{-j} \xi  y+ \tau) dy$, so that for $n_2\geq 0$ 
$$
(N^{j} D_\xi)^{n_2} \lambda_{j,a}(\xi,\tau)
=  \int   \tsigma( y)  y^{n_2}  \,
(D^{n_2} \widehat{\rho} \,) (N^{-j} \xi  y+ \tau) dy.
$$
For a fixed $\xi$, this integral as a function of $\tau$ is a convolution of two Schwartz functions whose
Schwartz seminorms are bounded uniformly in $j\in\NN$ and $\xi$ in the range
$|N^{-j} \xi|\sim_N 1$. Hence for any $M'>0$ we have
\begin{equation}\label{garbage-e2}
| (N^{j} D_\xi)^{n_2} \lambda_{j,a}(\xi,\tau)
\lesssim_{n_2,N,M'}  (1+|\tau|)^{-M'}.
\end{equation}

\medskip
\item 
We claim that for $n_1\geq 0$ and $\xi$ in the indicated range,
\begin{equation}\label{garbage-e1}
(N^j D_\xi)^{n_1}   \frac{ \langle a\xi - \tau  \rangle^{\gamma} }{N^{j\gamma}}
\lesssim_{n_1,N} \langle \tau \rangle.
\end{equation}
Indeed, we have 
$$(N^j D_\xi)^{m_1}   \frac{ \langle  a\xi - \tau \rangle^{\gamma} }{N^{j\gamma}}
\lesssim  \frac{ \langle  a\xi - \tau  \rangle^{\gamma-m_1} }{ N^{- j(\gamma-m_1)}  }.
$$
If $m_1=0$, we write
$$
\frac{ \langle  a\xi  - \tau \rangle }{N^{j}} 
\leq  \frac{ \langle \tau  \rangle }{N^{j}} + \frac{ \langle a\xi  \rangle }{N^{j}}
\lesssim \langle \tau  \rangle + O(N^2) \lesssim_N \langle \tau  \rangle.
$$
and the claim follows.
If $n_1\geq 1$, the exponent $\gamma-n_1$ is
negative, so we need the estimate 
\begin{equation}\label{bam2}
N^j \langle a\xi - \tau \rangle^{-1} \lesssim 1+|\tau|.
\end{equation}
If $|\tau|\leq \frac{1}{4} N^j$, then for $|\xi|\geq \frac{1}{2}N^j $
we have
$$
N^{-j} \langle a\xi  - \tau \rangle \geq  N^{-j} | a\xi - \tau  | \geq N^{-j} \Big( \frac{1}{2}N^j  - \frac{1}{4} N^j \Big)
= \frac{1}{4},
$$
so that $N^j \langle \tau-  a\xi  \rangle^{-1} \lesssim 1$. 
If on the other hand $|\tau|\geq \frac{1}{4} N^j$, then
$$
N^{j} \langle  a\xi - \tau  \rangle^{-1} \leq  N^j \lesssim |\tau|
$$
and the claim again is proved.

\end{itemize}

The proof of (\ref{k-decay-t}) is similar, except that instead of $L_M$ we use 
$L'_M=(1+ |t|^2 )^{-M} (1+ D_\tau^2)^M $ with $L_M e^{2\pi i t\tau } = e^{2\pi i t\tau }$
and integrate by parts in $\tau$. The details are omitted. 

\end{proof}

\begin{corollary}\label{cor-falp}
For $f\in\cals(\RR)$, we have the estimate
\begin{equation}\label{falp}
\|F_{j,a}f\|_{L^p(dxdt)} \lesssim_N N^{j(\gamma-\alpha)} \|f\|_{L^p(dx)}.
\end{equation}
\end{corollary}

\begin{proof}
We write $\|F_{j,a}f\|_{L^p(dxdt)}= \big\| \,\|F_{j,a} f(\cdot, t)\|_{L^p(dx)} \big\|_{L^p(dt)}$. Writing out
$F_{j,a}f$ as in (\ref{fajfa-2}), we see that it suffices to prove an estimate of the form
$$
\left\| \int  f(y)  {K}_{j,a}(x-y,t) dy\right\|_{L^p(dx)} \lesssim_N (1+|t|)^{-2} \|f\|_p.
$$
By Young's inequality, it suffices to prove that
$$
\int   |{K}_{j,a}(x,t) | dx \lesssim_N (1+|t|)^{-2} .
$$
But this is an easy consequence of (\ref{k-decay}) and (\ref{k-decay-t}).

\end{proof}

In the next proposition, 
we let $\phi \in C_c^\infty([-2,2])$ be a function such that $0\leq\phi \leq 1$ and $\phi  (\tau)=1$ for $|\tau|\leq 1$.
This can be the same function that we used to define the cut-offs in $|\xi|$.
The small number $\epsilon>0$ will be fixed later.

\begin{proposition}\label{fourier-local}
Let
\begin{align*}
&m_{j,a}^{\rm main} (\xi,\tau)= m_{j,a}(\xi,\tau) \phi(N^{-j\epsilon} \tau),
\\
&K_{j,a}^{\rm main}(x,t)=\iint e^{2\pi i ((x-at) \xi + t\tau )} m_{j,a}^{\rm main} (\xi,\tau) d\xi d\tau,
\\
&F_{j,a}^{\rm main} f (x,t)= N^{j(\gamma-\alpha)}
 \int  f(y)  K_{j,a}^{\rm main} (x-y,t) dy,\ \ f\in\cals(\RR).
\end{align*}
Then for all $1\leq p \leq \infty$ and $M''\in\NN$ we have
\begin{equation}\label{f-main}
\| F_{j,a}f  - F_{j,a}^{\rm main} f \|_{L^p(\RR^2) }
\lesssim_{N,M''} N^{-j\epsilon M''} \|f\|_p,
\end{equation}
where the implicit constant may depend on $p$, $\epsilon$, $N$, and $M''$, but not on $j$.
\end{proposition}
\medskip

\noindent{\bf Remark.}
In the original (independent of $a$) Fourier coordinates $(\xi,s)$, we have 
$$
F_{j,a}^{\rm main} f  = \mathcal{F}^{-1} \left[ \tilde{m}_{j,a}^{\rm main}(\xi,s) \widehat{f} (\xi) \right],
$$
where $\tilde{m}_{j,a}^{\rm main}(\xi,s)= N^{j(\gamma-\alpha)} m_{j,a}^{\rm main} (\xi,s+a\xi)$
is supported in the set
\begin{align*}
\mathcal{K}_{j,a}^\epsilon 
&= \left\{ (\xi,s)\in\RR^2:\ \   
\frac{1}{2}N^j \leq |\xi| \leq 4 N^{j+2},\ \ |s + a \xi | \leq 2N^{j\epsilon} \right\}.
\end{align*}

\medskip

\begin{proof}
The remark after the proposition follows immediately upon changing coordinates. We now prove the proposition,
We have
$$
F_{j,a}f  - F_{j,a}^{\rm main} f (x,t) = N^{j(\gamma-\alpha)}
 \int  f(y)  (K_{j,a}- K_{j,a}^{\rm main} )(x-y,t) dy.
$$
As in Corollary \ref{cor-falp}, it suffices to prove that
$$
 \int  |   (K_{j,a}- K_{j,a}^{\rm main} )(x,t)| dx \lesssim_{N,M''} N^{-j\epsilon M''} (1+|t|)^{-2}.
$$
This in turn follows from estimates analogous to (\ref{k-decay}) and (\ref{k-decay-t}), namely
\begin{equation}\label{k-decay-tail}
\big|   (K_{j,a}- K_{j,a}^{\rm main} )(x,t)\big| 
\leq C_{M,M'',N} N^{-j\epsilon M''} 
\Big(1+ N^{2j}|x- ta|^2 + |t|^2 \Big)^{-M}.
\end{equation}
with $C_{M,M'',N}$ independent of $j$ and $a$.  
To prove this, we proceed as in the proof of  (\ref{k-decay}) and (\ref{k-decay-t}), with the following modifications.
We have
\begin{equation*}
(K_{j,a}- K_{j,a}^{\rm main} )(x,t)
= \iint e^{2\pi i ((x-at) \xi + t\tau )} \left[ m_{j,a}(\xi,\tau)- m_{j,a}^{\rm main} (\xi,\tau) \right] d\xi d\tau,
\end{equation*}
with 
$$
m_{j,a}(\xi,\tau)- m_{j,a}^{\rm main} (\xi,\tau)
=  \frac{ \langle a\xi - \tau \rangle^{\gamma} }{N^{j\gamma}}  |\lambda_j(\xi,\tau)| 
\Big(1- \phi(N^{-j\epsilon} \tau)\Big) .
$$
We integrate by parts as in the proof of  (\ref{k-decay}) and (\ref{k-decay-t}), but also use that 
$1- \phi(N^{-j\epsilon} \tau)$ is supported in $|\tau|\geq N^{j\epsilon}$, so that we can separate out
factors 
$(1+|\tau|)^{-M''}\lesssim N^{-j\epsilon M''}$ from the estimate (\ref{garbage-e2}) and from the analogous 
estimate for $D_\tau$ before proceeding with the rest of the argument.

\end{proof}


\section{Decoupling for the Cantor bush}\label{sec-decoupling}


\subsection{Preliminaries}
We will need to develop decoupling inequalities for functions with Fourier support contained in a neighbourhood
of the Cantor bush $\bigcup_{a\in A_j } \mathcal{K}_{j,a}^\epsilon$, with $\mathcal{K}_{j,a}^\epsilon$
defined in the remark after Proposition \ref{fourier-local}.
We use parts of the decoupling machinery developed by Bourgain and Demeter \cite{BD2015}, \cite{BD-expo}. The notation below will follow the conventions of \cite{BD-expo}, with minor modifications.
We will also rely on a 1-dimensional Cantor decoupling inequality
proved in \cite{Laba-Wang}.

For $L>0$, an {\it $L$-interval} in $\RR$ will be an
interval of length $L$ with endpoints in $L\ZZ$. If a coordinate system in $\RR^2$ is given, 
an {\it $L$-square}  will be a $2$-dimensional square of side length $L$, with vertices in $L\ZZ^2$ and sides parallel to the coordinate axes. We will often use $L=N^k$ with $k\in\ZZ$; in that case, any $N^k$-squate $Q$ and any $N^{k'}$-square $Q'$ 
in the same coordinate system
are either nested or disjoint except possibly for an edge or vertex. Unless stated otherwise, we will assume all $L$-squares to be closed. 

Note that the definition above relies on a fixed choice of a coordinate system. 
In the inductive arguments below, we will use many coordinate systems corresponding to different
portions of the Cantor set. 
We will say that two such coordinate systems are {\it compatible} if a $1$-square in one coordinate system can be covered by $O(1)$ 
1-squares in the other coordinate system, and vice versa, with the $O(1)$ constants independent of $N,j,k$.

We will use local weights in 1 and 2 dimensions. 
If $R$ is the rectangle $\{(x,t):\ |x-x_0|\leq r_x,\ |t-t_0|\leq r_t\}$, we define
\begin{equation}\label{def-weight}
w_R(x,t)=\left(1+{\sqrt{\Big(\frac{x-x_0}{r_x}\Big)^2+\Big(\frac{t-t_0}{r_t}\Big)^2} }\right)^{-100}
\end{equation}
and, for a locally integrable function $g:\RR^2\to \CC$,
$$
\|g\|_{L^p (w_R)}=\left(  \int |g|^p w_R\right)^{1/p}.
$$
In dimension 1, if $I$ is the interval $x_0-r\leq x \leq x_0+r$, we define
$$
w_I(x)=\left(1+\frac{|x-x_0|}{r}\right)^{-1000},
$$
and
$\|g\|_{L^p (w_I)}$ is defined similarly. Typically, $R$ and $I$ will be $N^k$-squares and 
intervals. We will use $Q,R, S$ for squares and $I,J$ for intervals;
this will also indicate whether the associated weight $w$ is taken in 1 or 2 dimensions. 


We will use repeatedly the following covering argument (cf. \cite[Lemma 4.1]{BD-expo}).

\begin{lemma}\label{lemma-covering}
Let $\RR^2=\bigcup_{Q\in\mathcal{Q}} Q$ be a covering of the plane by $L$-squares associated with some
coordinate system. Then we have the following estimates, with the implicit constants independent of $L$.

\medskip
(a) $\sum_{Q\in\mathcal{Q}} w_Q \sim 1$

\medskip
(b) $\min_{x\in Q} w_Q(x) \sim \max_{x\in Q} w_Q(x)$

\medskip
(c) Let $\RR^2=\bigcup_{R\in\mathcal{R}} R$ be a covering of the plane by $L'$-squares in a 
possibly different but compatible coordinate system, with $L'\sim L$. Suppose that $\{g_i\}_{i\in\mathcal{I}}$ is a finite family of
functions such that for $g=\sum g_i$, and for every $R\in\mathcal{R}$, we have the estimate
$$
\|g\|_{L^p(R)} \leq K \Big( \sum_i \|g_i\|^2_{L^p(w_R)} \Big)^{1/2}.
$$
Then we also have
$$
\|g\|_{L^p(w_Q)} \lesssim K \Big( \sum_i \|g_i\|^2_{L^p(w_Q)} \Big)^{1/2}
$$
for all $Q\in\mathcal{Q}$.
\end{lemma}

\begin{proof}
Parts (a) and (b) are clear from the definition of $w_Q$. We now prove (c). For a given $Q$, let 
$c_R=\max_{x\in R}w_Q(x)$. We claim that
\begin{equation}\label{chain}
w_Q \leq \sum_R c_R \one_R \lesssim \sum_R c_R w_R \lesssim w_Q.
\end{equation}
The first two inequalities in (\ref{chain}) are clear; we need to verify the last one.
Using (b) and then (a), we have
$$
\sum_R c_R w_R(x)  = \sum_R \max_{y\in R} w_Q(y) \cdot w_R(x) 
\lesssim \sum_R w_Q(x) w_R(x) \lesssim w_Q(x),
$$
as required.

With $c_R$ as above, we write
\begin{align*}
\|g\|^2_{L^p(w_Q)} & = \left[ \int |g|^p w_Q \right]^{2/p}
\lesssim  \left[ \sum_R c_R \int_R |g|^p  \right]^{2/p}
\\
& =  \left[ \sum_R c_R \|g\|_{L^p(R)}^p  \right]^{2/p}
\\
& \lesssim K^2  \left[ \sum_R \Big( c_R^{2/p} \sum_i \|g_i \|_{L^p(w_R)}^2 \Big)^{p/2}  \right]^{2/p}
\\
& \lesssim K^2  \sum_i \left[ \sum_R \Big( c_R^{2/p}  \|g_i \|_{L^p(w_R)}^2 \Big)^{p/2}  \right]^{2/p},
\end{align*}
where at the last step we used Minkowski's inequality in $\ell^{p/2}(\mathcal{R})$. Continuing the calculation
and using (\ref{chain}) at the end, we get
\begin{align*}
\|g\|^2_{L^p(w_Q)}
& \lesssim  K^2 \sum_i \Big[ \sum_R  c_R  \|g_i \|_{L^p(w_R)}^p  \Big]^{2/p}
\\
&= K^2 \sum_i \Big[ \int \sum_R  |g_i|^p c_R w_R  \Big]^{2/p}
\\
&\lesssim  K^2 \sum_i \Big[ \int   |g_i|^p w_Q  \Big]^{2/p} = K^2 \sum_i \|g_i\|^2_{L^p(w_Q)}
\end{align*}
as claimed.

\end{proof}


\subsection{Decoupling for Cantor strips}

Let $S\subset[N-1]$ be a $\Lambda(p)$ set obeying (\ref{lambda-p-0}) and (\ref{largeA}), with $N$ sufficiently large to be determined later.
Our basic tool, borrowed from \cite[Lemma5]{Laba-Wang}, is the
following single-scale decoupling inequality which follows from Bourgain's $\Lambda(p)$ estimate. We will need
a slightly modified version with intervals of length $2$ instead of 1. This is easy to arrange using a partition of unity,
cf. the remark before Lemma 5 in \cite{Laba-Wang}.

\begin{lemma}\label{dtc-lemma} 
With $S$ as above,
let $E=S+[0,2]$, and let $h:\RR\to\CC$ be a locally integrable function with $\widehat{h}$ supported on 
$E$. Let $h=\sum_{a\in S} h_a$, where $\widehat{h_a}$ is supported on $a+[0,2]$. Then for $2\leq p \leq p_0$ we have
\begin{equation}\label{decoupling}
\|h\|^2_{L^p(w_I)}\leq C_1^2 \sum_{a\in S} \|h_a\|^2_{L^p(w_I)}
\end{equation}
for any 1-interval $I$, where $C_1$ depends on $p$ but not on $N$ or $h$.
\end{lemma}

We need to extend the estimate (\ref{decoupling}) to 2-dimensional product sets consisting of parallel strips
corresponding to the Cantor intervals.

\begin{lemma}\label{dtc-lemma2} 
With $E\subset \RR$ defined above and $L_1<L_2$,
let $h:\RR^2\to\CC$ be a locally integrable function such that $\widehat{h}$ is supported on 
$E\times [L_1,L_2]$. Assume that $h=\sum_{a\in S}h_a$, lwhere $\widehat{h_a}$ is supported on $[a,a+2]\times[L_1,L_2]$. Then 
\begin{equation}\label{decoupling2}
\|h\|^2_{L^p(w_Q)} \lesssim C_1^2 \sum_{a\in S} \|h_a\|^2_{L^p(w_Q)}
\end{equation}
for any 1-square $Q$.
\end{lemma}

\begin{proof}
We will prove that for any $t_1<t_2$, and for any $1$-interval $I$, we have
\begin{equation}\label{parallel1}
\int_{t_1}^{t_2} \int |h(x,t)|^p w_I(x) \,dx\,dt\leq C_1^2 \left[ \sum_{a\in S} \Big( \int_{t_1}^{t_2} \int |h_a(x,t)|^p w_I(x)\,dx\, dt \Big)^{2/p} \right]^{1/2}.
\end{equation}
Then (\ref{decoupling2}) follows from Lemma \ref{lemma-covering} (c).

Consider the function $h^{(t)}(x):=h(x,t)$ as a function of $x$, with $t$ fixed. We have
\begin{align*}
h(x,t)&= \int \Big[ \int_{L_1}^{L_2} \widehat{h}(\xi,s)e^{2\pi i ts}ds \Big] e^{2\pi i \xi x} d\xi,
\end{align*}
hence $\widehat{h^{(t)}}(\xi) = \int_{L_1}^{L_2} \widehat{h}(\xi,s)e^{2\pi i ts}ds $ is supported on $E$, and
satisfies the assumptions of Lemma \ref{dtc-lemma} with $(h^{(t)})_a(x)=h_a(x,t)$. 
By (\ref{decoupling}), we have
\begin{equation}\label{dtc-3}
\|h(x,t)\|^2_{L^p(w_I(x))}\leq C^2 \sum_{a\in S} \|h_a(x,t)\|^2_{L^p(w_I(x))}.
\end{equation}
Let $H_a(t)=\|h_a(x,t)\|^2_{L^p(w_I(x))}$,
then by (\ref{dtc-3}) and Minkowski's inequality we have
\begin{align*}
\int_{t_1}^{t_2} \int |h(x,t)|^p & w_I(x)\,dx\,dt = \int_{t_1}^{t_2} \|h(x,t)\|^p_{L^p(w_I(x))} dx\,dt 
\\
& \leq C_1^p  \int_{t_1}^{t_2}  \Big[ \sum_a H_a(t) \Big]^{p/2}  dxdt 
\\
& = C_1^p  \Big\| \sum_a H_a(t) \Big\|_{L^{p/2}([t_1,t_2])}^{p/2}
\\
& \leq C_1^p  \Big[ \sum_a \|  H_a(t) \|_{L^{p/2}([t_1,t_2]) } \Big]^{p/2}
\\
& = C_1^p  \left[ \sum_a \Big( \int_{t_1}^{t_2} \|h_a(x,t)\|_{L^p(w_I(x))}^p dt \Big)^{2/p} \right]^{p/2},
\end{align*}
which proves (\ref{parallel1}). 

\end{proof}


\subsection{Local coordinates adjusted to the Cantor bush}

The key geometric observation is that, for each $a\in A_k$, the Cantor branches corresponding to $a'\in A_{k+1,a}$ 
in the next iteration 
can be treated as parallel when restricted to segments of somewhat shorter length. More precisely, the corresponding part of the Cantor bush can be covered efficiently by a rescaled and rotated copy of the set $\cale$ from Lemma \ref{dtc-lemma2}.

\begin{lemma}\label{coordinates}
Let $a\sim 1$. Define new coordinate systems $(u,v)$ on $\RR^2_{x,t}$ and $(\eta,\tau)$ on $\RR^2_{\xi,s}$:
\begin{equation}\label{cc-e1}
\begin{split}
&u=\frac{x-at}{1+a^2},\ \ v=\frac{ax+t}{1+a^2},
\\
&\eta=\xi-as, \ \ \tau=a\xi+s.
\end{split}
\end{equation}
Then:

\smallskip

(a) We have $(u,v)^T=\mathbb{A}(x,y)^T$ and $(\eta,\tau)^T=(\mathbb{A}^T)^{-1}(\xi,s)^T$, where 
$\mathbb{A}$ is an orthogonal matrix.
Hence the coordinate systems $(u,v)$ and $(\eta,\tau)$ are orthogonal and dual to each other.

\smallskip

(b) Let $0<\xi_1<\xi_2$, $0<\Delta a$, and $S\subset [N-1]$. Assume that
\begin{equation}\label{cc-e2}
\xi_2-\xi_1 \leq \frac{\xi_1}{N}.
\end{equation}
Then the set 
$$\mathcal{K}=\mathcal{K}[a,\Delta a, \xi_1,\xi_2] 
=\bigcup_{b\in S} \mathcal{K}_b,
$$
where
$$
\mathcal{K}_b:= \left\{(\xi,  s):\ \xi_1\leq \xi \leq \xi_2,\ 
-\frac{s}{\xi} \in a + \frac{\Delta a}{N}\Big(b+[0,1] \Big)
\right\}$$
is contained in
$$ \mathcal{E}:= \left\{(\eta,\tau):\ \xi_1(1+a^2) \leq \eta \leq \xi_2(1+a^2+a\Delta a),\ 
\tau \in -\xi_1 \frac{\Delta a}{N}(S+[ 0, 2 ] )
\right\},$$
with the individual branches $\mathcal{K}_b$ contained in the corresponding strips 
$\mathcal{E}_b:=\{(\eta,\tau)\in\mathcal{E}: \tau\in-\xi_1 \frac{\Delta a}{N}\big(b+[ 0, 2 ] \big)\}$
of $\mathcal{E}$.
\end{lemma}

\begin{proof} 
Part (a) is easily verified by direct calculation. We now turn to (b). We need to prove the following: for
\begin{equation}\label{def-b}
b=\frac{\Delta a}{N}(b_0+\Delta b),\ \ b_0\in S,\ \ 0\leq \Delta b \leq 1,
\end{equation}
the line segment 
$J_b=\{(\xi,  s):\ \xi_1\leq \xi \leq \xi_2,\ 
\frac{s}{\xi} +a+b=0\}$ is contained in the set
\begin{equation}\label{def-strip}
 \left\{(\eta,\tau):\ \xi_1(1+a^2) \leq \eta \leq \xi_2(1+a^2+a\Delta a),\ 
\tau \in -\xi_1 \frac{\Delta a}{N}\big(b_0+[ 0, 2 ] \big)
\right\}
\end{equation}
Since the set (\ref{def-strip}) is convex, it suffices to prove this for the endpoints of $J_b$. A very short calculation
shows that these are given by $(\eta_i,\tau_i)$, $i=1,2$, where
$$
\eta_i = \xi_i (1+a^2+ab),\ \ \tau_i = -\xi_i b.
$$
For $b$ as in (\ref{def-b}), we have $0\leq b\leq \Delta a$. This clearly implies that $\eta_1,\eta_2$ satisfy 
the constraint in (\ref{def-strip}). Next, we have
$$
-\tau_1 =\xi_1 b \in  \xi_1 \frac{\Delta a}{N}\big(b_0+[ 0, 1 ] \big).
$$
Finally, we write 
$$
-\tau_2 = \xi_2 b = \xi_1 b + (\xi_2-\xi_1)b \in \xi_1 \frac{\Delta a}{N}\big(b_0+[ 0, 1 ] \big) + (\xi_2-\xi_1)b,
$$
and by (\ref{cc-e2}),
$$0<(\xi_2-\xi_1)b \leq \xi_1 \frac{\Delta a}{N},$$
which completes the proof.


\end{proof}

\begin{corollary}\label{cor-step0} 
Let $\mathcal{K}=\bigcup_{b\in S} \mathcal{K}_b$ be as in Lemma \ref{coordinates}.
For a locally integrable function $g:\RR^2_{x,t}\to\CC$ with $\widehat{g}$ is supported on $\mathcal{K}$,
let $\widehat{g_b}=\one_{ \mathcal{K}_b } \widehat{g}$. Then 
\begin{equation}\label{step0}
\|g\|^2_{L^p(w_Q)} \lesssim \sum_{b\in S} \|g_b\|^2_{L^p(w_Q)}
\end{equation}
for any $L$-square $Q$ with $L=N(\xi_1 \Delta a)^{-1}$. 
\end{corollary}

\begin{proof}
We change the coordinates as in Lemma \ref{coordinates}. The function $h(\eta,\tau)
= g(x(\eta,\tau),t(\eta,\tau))$ is Fourier supported in $\mathcal{E}$ and satisfies the assumptions of the lemma, with 
$h_b(\eta,\tau) = g_b(x(\eta,\tau),t(\eta,\tau))$ for $b\in S$. Note that the set $\mathcal{E}$ is a 
rescaled and reflected copy of the set $E\times [L_1,L_2]$ (with appropriate $L_1,L_2$) from Lemma \ref{dtc-lemma2}.
Applying (\ref{decoupling2}) to a scaled copy of $h$ and then undoing the scaling and the coordinate change, we get that (\ref{step0}) holds
with $Q$ replaced by any $L$-square in the $(\eta,\tau)$ coordinates. To pass to $L$-squares in
the $(x,t)$ coordinates, we use
Lemma \ref{lemma-covering}.

\end{proof}


\subsection{The inductive argument for the Cantor bush}

\begin{lemma}\label{partition}
Let $\mathcal{K}_{j}^\epsilon= \bigcup_{a\in A_j} \mathcal{K}_{j,a}^\epsilon$, where
\begin{align*}
\mathcal{K}_{j,a}^\epsilon 
&= \left\{ (\xi,s)\in\RR^2:\ \   
\frac{1}{2}N^j \leq |\xi| \leq 4 N^{j+2},\ \ |s + a \xi | \leq 2N^{j\epsilon} \right\}.
\end{align*}
Then $\mathcal{K}_{j}^\epsilon$ can be covered by $O(N^{-2+j\epsilon})$ finitely overlapping sets of the form
$\mathcal{K}_j^0= \bigcup_{a\in A_j} \mathcal{K}_{j,a}^0$, where
$$
\mathcal{K}_{j,a}^0= \mathcal{K}_{j,a}^0[a_0,\xi_1,\xi_2]
= \left\{ (\xi,s)\in\RR^2:\ \   
\xi_1 \leq |\xi| \leq \xi_2,\ \ -\frac{s}{\xi}\in a_0 + a+[0,N^{-j}]
 \right\}
$$
with 
$\frac{1}{4}N^j \leq \xi_1 \leq 4 N^{j+2}$,
$\frac{\xi_1}{2N}  \leq \xi_2-\xi_1 \leq \frac{\xi_1}{N}$, $a_0\in \frac{1}{2}N^{-j}\ZZ$ and $|a_0|=O(N^{-j+j\epsilon})$.

\smallskip
Furthermore, if $G_a\in L^p(\RR^2)$ is Fourier-supported in $\mathcal{K}_{j,a}^\epsilon$, then there is a decomposition
\begin{equation}\label{partition-Lp}
G_a=\sum_{i\in\mathcal{I}} g_a^{(i)},
\end{equation}
where the summation runs over a set $\mathcal{I}$ of cardinality $O(N^{-2+j\epsilon})$,
each $g_a^{(i)}$ is Fourier supported in some set $\mathcal{K}_{j,a}^0 [a_0,\xi_1,\xi_2]$ as above,
and $\| g_a^{(i)}\|_p\lesssim_N \|G_a\|_p$ with the constant independent of $j$.

\end{lemma}

\begin{proof}
We first cover the strip 
$\frac{1}{2}N^j \leq |\xi| \leq 4 N^{j+2}$ by $O(N^2)$ strips $\xi_1 \leq |\xi| \leq \xi_2$ as indicated.
Suppose now that $(\xi,s)\in \mathcal{K}_{j,a}^\epsilon$ for some $a\in A_j$. Then
$$
\left| \frac{s}{\xi}+a \right| \leq \frac{2N^{j\epsilon}}{|\xi|} \lesssim N^{-j+j\epsilon},
$$
so that $-\frac{s}{\xi}\in a_0 +a+[0,N^{-j}] \subset a_0+E_j$ for some $a_0$ as in the statement of the lemma.

For the second claim, allowing overlaps in the covering of $\mathcal{K}_{j,a}^\epsilon$,
we can associate with it a smooth partition of unity $\{\Xi^{(i)}_a\}_{i\in\mathcal{I}}$ such that
$\sum_{i\in\mathcal{I}} \Xi^{(i)}_a = 1$ on $\mathcal{K}_{j,a}^\epsilon$ and
$\big\| (\Xi^{(i)}_a )^\vee \big\|_1=O_N(1)$ uniformly in $j$ and $a$. Then the functions 
$g_a^{(i)}= G_a* ( \Xi^{(i)}_a)^\vee $ satisfy the desired conclusions, with the $L^p$ estimate following
from Young's inequality. 
\end{proof}

\begin{proposition}\label{multidecoupling}
Let $\mathcal{K}_{j}^0[a_0,\xi_1,\xi_2]$ be as in Lemma \ref{partition}. For a function $g:\RR^2\to \CC$ with 
$\supp \widehat{g}\subset \mathcal{K}_{j}^0[a_0,\xi_1,\xi_2]$, and for $k=1,\dots,j$, write $g=\sum_{a\in A_k} g_{k,a}$ with $\widehat{g_{k,a}}$ supported in the set
\begin{equation}\label{subbush}
\left\{ (\xi,s)\in\RR^2: \xi_1\leq \xi\leq \xi_2 ,\ -\frac{s}{\xi} \in a_0+a+[0,N^{-k}] \right\}.
\end{equation}
(Note that this defines $g_{k,a}$ uniquely since the sets above are disjoint for different $a\in A_k$.)
Then there is a constant $C_2$ (independent of $N,k,j$) such that for any $N^k L$-square $Q_k$ with $L =\xi_1^{-1}$,  we have 
\begin{equation}\label{multi-e}
\Big( \sum_{S\in\mathcal{T}(Q_k)} \|g \|^p_{L^p(w_S)}\Big)^{1/p} 
\leq  C_2 ^{k} \Big(\sum_{a\in A_k} \| g_{k,a}\|_{L^p(w_{Q_k})}^2\Big)^{1/2},
\end{equation}
where $Q_k=\bigcup_{S\in\mathcal{T}(Q_k)}S$ is a tiling of $Q_k$ by $L$-squares.
\end{proposition}

\begin{proof}
The proof is similar to the proof of Proposition 1 in \cite{Laba-Wang}
We proceed by induction in $k$, using Corollary \ref{cor-step0} at each step.
To initialize, we write using the notation from Lemma \ref{coordinates}
$$\mathcal{K}_{j}^0[a_0,\xi_1,\xi_2] \subset \mathcal{K}[a_0,1,\xi_1,\xi_2].$$
Applying Corollary \ref{cor-step0}, we get that
\begin{equation} \label{step1}
\| g \|^2_{L^p(w_{Q_1})}\leq C_3 \sum_{a\in A_1} \| g_{1,a}\|_{L^p(w_{Q_1})}^2 ,
\end{equation}
with $Q_1$ as above and some constant $C_3$ independent of $N,j$. Similarly, for $1\leq k\leq j-1$ and
$a\in A_k$, the set (\ref{subbush}) is contained in $\mathcal{K}[a_0+a, N^{-k}, \xi_1,\xi_2]$. 
Applying Corollary \ref{cor-step0} again, we get 
\begin{equation} \label{stepk}
\| g_{k,a}\|^2_{L^p(w_{Q_{k+1}})}\leq C_3 \sum_{b\in A_{k+1,a}} \| g_{k+1,b}\|_{L^p(w_{Q_{k+1}})}^2 .
\end{equation}

To put the inductive argument together, we use parallel decoupling. 
Let $C_4$ be a constant such that for all $N^kL$ squares $Q_k$ with $k\geq 1$,
\begin{equation}\label{weight-compare}
\sum_{S\in\mathcal{T}(Q_k)}w_S \leq C_4 w_{Q_k}.
\end{equation}
By the rapid decay of $w$, we can choose $C_4$ independent of $N$ and $k$.

We first note that (\ref{multi-e}) for $k=1$ is provided by (\ref{step1}), with the trivial tiling consisting of a single square.
Assume now that we have (\ref{multi-e}) for some $k$ with $1\leq k \leq j-1$. 
Let $Q_{k+1}$ be an $N^{k+1}L$-square, and let $Q_{k+1}=\bigcup_{Q_k\in \mathcal{J}} Q_k$ be a tiling of $Q_{k+1}$ by $N^kL$-squares. By the inductive assumption (\ref{multi-e}), Minkowski's inequality in $\ell^{p/2}(\mathcal{J})$, a rescaling of (\ref{weight-compare}), and (\ref{stepk}), in that order, we have
\begin{equation}\label{Minkowski}
\begin{split}
\sum_{S\in\mathcal{I}(Q_{k+1})} \|g \|^p_{L^p(w_S)}
&= \sum_{Q_{k}\in\mathcal{J}} \sum_{S\in\mathcal{T}(Q_k)} \| g \|^p_{L^p(w_S)}\\
&\leq  C_2^{kp} \sum_{Q_k\in\mathcal{J}} \Big( \sum_{a\in A_k} \| g_{k,a}\|_{L^p(w_{Q_k})}^2 \Big)^{p/2}
\\
&\leq  C_2^{kp} \left[ \sum_{a\in A_k} \Big( \sum_{Q_k\in\mathcal{J}} \|g_{k,a}\|_{L^p(w_{Q_k})}^p \Big)^{2/p} \right]^{p/2}
\\
&\leq  C_2^{kp} C_4 \left[ \sum_{a\in A_k}  \| g_{k,a}\|_{L^p(w_{Q_{k+1}})}^2  \right]^{p/2}
\\
&\leq  C_2^{kp} C_4 C_3^p \left[ \sum_{a\in A_{k+1}}  \|f_{k+1,a}\|_{L^p(w_{Q_{k+1}})}^2  \right]^{p/2}.
\end{split}
\end{equation}
This proves (\ref{multi-e}) with $C_2=C_3^{1/2}C_4^{1/p}$.

\end{proof} 

\begin{corollary}\label{decoupling-global}
Let $g=\sum_{a\in A_j} g_{j,a}$ be as in Proposition \ref{multidecoupling}.
Then
\begin{equation}\label{multi-global}
\|g\|_{L^p(\RR^2)} 
\lesssim  C_2 ^{j} 
\Big(\sum_{a\in A_j} \| g_{j,a}\|_{L^p(\RR^2)}^2\Big)^{1/2}.
\end{equation}
\end{corollary}

\begin{proof} We use the notation from the proof of Proposition \ref{multidecoupling}.
Note first that by 
(\ref{multi-e}), we have
\begin{equation}\label{boring-1}
\|g\|_{L^p(Q_j)} \lesssim  \Big( \sum_{S\in\mathcal{T}(Q_j)} \|g \|^p_{L^p(w_S)}\Big)^{1/p} 
\leq  C_2 ^{j} \Big(\sum_{a\in A_j} \| g_{j,a}\|_{L^p(w_{Q_j})}^2\Big)^{1/2}.
\end{equation}
Let now $Q$ be an $MN^j L$-square, and let $Q=\bigcup_{Q_j\in \mathcal{J}} Q_j$ be a tiling of $Q$ by $N^jL$-squares. Using (\ref{boring-1}) and Minkowski's inequality in $\ell^{p/2}(\mathcal{J})$ as in (\ref{Minkowski}), we get
\begin{align*}
\|g\|_{L^p(Q)}^p &= \sum_{Q_j\in\mathcal{J}} \|g\|_{L^p(Q_j)}^p
\\
& \lesssim  
 C_2^{jp} \sum_{Q_j\in\mathcal{J}} \Big( \sum_{a\in A_j} \| g_{j,a}\|_{L^p(w_{Q_j})}^2 \Big)^{p/2}
\\
&\leq  C_2^{jp} \left[ \sum_{a\in A_j} \Big( \sum_{Q_j\in\mathcal{J}} \|g_{j,a}\|_{L^p(w_{Q_j})}^p \Big)^{2/p} \right]^{p/2}
\\
&\lesssim  C_2^{jp} \left[ \sum_{a\in A_j}  \| g_{j,a}\|_{L^p(\RR)}^2  \right]^{p/2},
\end{align*}
where the last step follows from Lemma \ref{lemma-covering} (a). Since this holds for any $Q$ with the constant
independent of $M$ and $Q$, we have proved (\ref{multi-global}).
\end{proof}

\begin{corollary}\label{decoupling-usable}
Let $G=\sum_{a\in A_j} G_a$,
where each $G_a\in L^p(\RR^2)$ is Fourier-supported in $\mathcal{K}_{j,a}^\epsilon$.
Then
\begin{equation}\label{multi-usable}
\|G\|_{L^p(\RR^2)} 
\lesssim_N  C_2 ^{j} N^{j\epsilon}
\Big(\sum_{a\in A_j} \| G_{a}\|_{L^p(\RR^2)}^2\Big)^{1/2}.
\end{equation}

\end{corollary}

\begin{proof}
This follows by using the decomposition in Lemma \ref{partition} and then applying Corollary \ref{decoupling-global}
to each 
$g^{(i)}=\sum_{a\in A_j} g_a^{(i)}$.

\end{proof}


\section{Proof of Theorems \ref{thm-main} and \ref{thm-average}}


\subsection{Proof of Theorem \ref{thm-main}}

Given $\delta>0$, choose $\epsilon>0$ small enough so that $5\epsilon<\delta$. Let $N$ be sufficiently large so that
$\alpha<\alpha'$ and $\frac{1}{p}-\frac{\alpha}{2}<\epsilon$. This is possible by (\ref{good-alpha}). Throughout
the proof, we may increase $N$ further as needed without changing the other parameters of construction.

By Lemma \ref{lemma-single-scale}, it suffices to prove the single scale maximal estimate (\ref{e-m-restricted}) with $\beta=4\epsilon$, i.e.
\begin{equation}\label{conclude-e1}
\|\tilde\calm f\|_p\leq C_N N^{4j\epsilon} \|f\|_p,\ \ j\geq j_0,
\end{equation}
for all $f\in\cals$ with $\supp \widehat{f}\subset \{N^j\leq|\xi| \leq 2N^{j+1}\}$, with the constant $C_N$
independent of $j$. 
By Lemma \ref{lemma-Fjsuffices}, this will follow if we can prove that the operators $F_{j}$
with $\gamma = (1/p)+\epsilon$ obey 
\begin{equation}\label{conclude-e2}
\|F_{j} f\|_{L^p(dxdt)} \lesssim_N N^{4j\epsilon} \|f\|_{L^p(dx)},\ \ f\in\cals.
\end{equation}

We start with the decomposition $F_j f=\sum_{a\in A_j} F_{j,a} f$ as in Section
\ref{localization}, and an application of Proposition \ref{fourier-local}. By (\ref{f-main}) with $M$ large enough so that 
$\epsilon M > 2$, we have
\begin{equation}\label{main-1}
\begin{split}
\|F_jf\|_p &= \Big\| \sum_{a\in A_{j}} F_{j,a} f \Big\|_p
\\
& \leq \Big\|\sum_{a\in A_{j}} F_{j,a}^{\rm main} f \Big\|_p + O_N( N^{-j} )\|f\|_p.
\end{split}
\end{equation}
where each $F_{j,a}^{\rm main} f$ is Fourier supported in $\mathcal{K}_{j,a}^\epsilon$.
Applying Corollary \ref{decoupling-usable} with $G_a= F_{j,a}^{\rm main} f$, we get 
$$
\Big\|\sum_{a\in A_{j}} F_{j,a}^{\rm main} f \Big\|_p
\lesssim_N N^{2 j\epsilon} \Big(\sum_{a\in A_j} \|
F_{j,a}^{\rm main} f \|_{p}^2\Big)^{1/2}.
$$
Using (\ref{f-main}) again, and then (\ref{falp}), we conclude that
\begin{equation}\label{main-2}
\begin{split}
\|F_jf\|_p
&\lesssim_N N^{2 j\epsilon} \Big(\sum_{a\in A_j} \|
F_{j,a} f \|_{p}^2\Big)^{1/2} + O( N^{-j} )\|f\|_p
\\
&\lesssim_N N^{2 j\epsilon} \Big(\sum_{a\in A_j} N^{j(\gamma-\alpha)} \|
 f \|_{p}^2\Big)^{1/2} + O( N^{-j} )\|f\|_p
\\
&\lesssim_N N^{ j(\gamma-\frac{\alpha}{2} + 2\epsilon) } \|f\|_p.
\end{split}
\end{equation}
We have
$$
\gamma-\frac{\alpha}{2} + 2\epsilon = \frac{1}{p}+\epsilon -\frac{\alpha}{2} + 2\epsilon
\leq 4\epsilon,
$$
so that (\ref{conclude-e2}) holds as claimed. This ends the proof of the theorem.


\subsection{Proof of Theorem \ref{thm-average}}

This is a minor modification of the above. To prove (\ref{lol}), define $F_j$ as in the proof of
Theorem \ref{thm-main}, but now use $\gamma=\frac{1}{p}-\frac{1}{r}+\epsilon$ instead.
Then the exponent at the end of the analogue of (\ref{main-2}) is
 $$
\gamma-\frac{\alpha}{2} + 2\epsilon = \frac{1}{p}-\frac{1}{r}+\epsilon -\frac{\alpha}{2} + 2\epsilon
\leq 4\epsilon-\frac{1}{r},
$$
which is negative if $\epsilon$ is small enough. The estimate (\ref{lol}) follows upon summing up in $j$
and then applying the Sobolev embedding theorem as in Section \ref{sobolev}.

The proof of (\ref{lolsob}) is similar, but with the factor $\sgamma$ in the definition of $F_j$ replaced by
$\langle \xi\rangle^\gamma$. Sobolev's embedding theorem is not needed for this part. 



\vskip.5in

\section{Acknowledgements}
The author was supported by the NSERC Discovery Grant 22R80520, and would like to thank
Malabika Pramanik, Andreas Seeger, Pablo Shmerkin and Joshua Zahl for helpful conversations.


\bigskip

\noindent{\sc Department of Mathematics, UBC, Vancouver,
B.C. V6T 1Z2, Canada}

\smallskip

\noindent{\it  ilaba@math.ubc.ca}

\end{document}